\documentclass[a4paper,11pt, twoside]{article}
\usepackage{amsmath,amsthm}
\usepackage{amssymb,latexsym}
\usepackage{mathrsfs}
\usepackage{enumerate}
\usepackage[colorlinks,citecolor=blue]{hyperref}
\usepackage[numbers,sort&compress]{usbib}
\usepackage{tcolorbox}

\newtheorem{theorem}{Theorem}[section]

\newtheorem{lemma}{Lemma}[section]

\newtheorem{definition}{Definition}[section]
\newtheorem{remark}{Remark}[section]

\newtheorem{example}{Example}[section]
\theoremstyle{definition} \theoremstyle{remark}
\numberwithin{equation}{section}
\allowdisplaybreaks
\numberwithin{equation}{section}

\newcommand{\R}{\mathbb{R}}

\DeclareMathAlphabet{\mathpzc}{OT1}{pzc}{m}{it}

\date{}

\begin{document}
	
	
	\date{}
	\title{{\bf Admissibility approach to nonuniform exponential dichotomies roughness with nonlocal perturbations
	}}
	
	\author{Jiawei He$^{1,2}$, Jianhua Huang $^{1,*}$ \\[1.8mm]
		\footnotesize  {$^1$  College of Science, National University of Defense Technology, Changsha 411100, China}\\
		\footnotesize  {$^2$ School of Mathematics, Guangxi University, Nanning 530004, China}
	}

	\maketitle
	
	\begin{abstract}
Nonuniform exponential dichotomy serves as an important characteristic of nonuniform hyperbolicity, while admissibility of function classes is often used to characterize nonuniform exponential dichotomy. In this paper, we investigate the preservation of nonuniform exponential dichotomy under certain nonlocal perturbations. By utilizing the concept of admissibility of a pair of function classes, we establish sufficient conditions to ensure that the dichotomy results are consistent with those in the homogeneous situation. These results need to satisfy a smallness integrability condition.
\\[2mm]
		{\bf Keywords:} Nonuniform exponential dichotomy, admissibility, roughness \\[2mm]
		{\bf 2020 MSC:} 34D09, 37D25
	\end{abstract}

\section{Introduction}
As is well known, the theory of exponential dichotomy for linear differential equations of the form $x'=A(t)x$ can be found in \cite{Massera1966,Coppel1978},  starting in the work of Perron \cite{Perron1930}. It is has been extensively developed in works such as differential equations \cite{Henry,Shen2021}, nonautonomous dynamical systems\cite{AAE2025,Sell2002,Popescu2006}, and invariant manifolds \cite{Nguyen2009,Li-Zeng-Huang2022} and many other domains \cite{Carvalho2025,Longo2019,Vu2024,Zhou21} etc..
In previous works, Pesin \cite{Pesin1976} generalized
the classical notion of uniform hyperbolicity through the nonuniform one in finite-dimensional systems, where a stable manifold theorem is derived for trajectories exhibiting nonuniform hyperbolicity.  Preda and Megan \cite{Preda1983} proposed the concept of nonuniform exponential dichotomy, avoiding any conditions on the norms associated with the dichotomy projections. The nonuniform exponential dichotomy formulates the nonuniform hyperbolicity for dynamical systems, motivated by ergodic theory and nonuniform hyperbolic theory, Barreira and Valls \cite{Barreira2008-1,Barreira2009-1,Barreira2011-1,Barreira2014-1} proved a series of improvements and contributions for nonuniform exponential dichotomy.
 For more details we refer the reader to
 the books \cite{Coppel1978,Chicone1999} and the papers \cite{Zhou16JFA,Preda2018,Wu2023} and references therein.

An effective approach to study the hyperbolicity for a system is the admissibility of a pair of function classes, which can  characterize the (nonuniform) exponential dichotomy. For the linear equation $x'=A(t)x$, $t\in J=\R,\R_+$ or $\R_-$, the pioneer work as Schaffer \cite{Schaffer1960} investigated the admissibility and exponential dichotomy,
Nguyen \cite{Nguyen2006} studied the admissibility to an admissible Banach space, where the existence and robustness of  exponential dichotomy on $\R_+$ were established. Barreira
and Valls \cite{Barreira2011-1} described nonuniform exponential dichotomy in terms of admissibility of function classes on $\R_+$ in
Banach spaces.
Sasu et al. \cite{Sasu2013admissibility} showed the  exponential dichotomy under the admissibility pair between $C_b$ and $L^p$ on $\R_+$.  Barreira et al. \cite{Barreira2017-1} characterized completely the notion of an  exponential dichotomy with respect to a family of norms in terms of the admissibility of bounded solutions.
Zhou et al. \cite{Zhou17} proved the relationship between the existence of nonuniform exponential dichotomy and the admissibility of a pair of Lynapunov bounded function classes, while
Dragivcevic et al. \cite{Dragivcevic2022} established the relationship between the existence of nonuniform exponential dichotomy and the admissibility of a pair of function classes
without a Lyapunov norm recently.
Note that without the notation of admissibility, Barreira and Valls \cite{Barreira2008-1}
 established the  roughness if $\|B(t)\|\leq \delta e^{-2\epsilon|t|}$ for $J=\R_-,\R_+$ or $\R$ with exponent index $\alpha>0$ and $\delta>0$ sufficiently small.
Recently, Pinto, Poblete and Xia \cite{Pinto2024} provided a generalization linear perturbations that satisfy an integral condition such that the roughness remains,  and the conclusions represent important improvements in weakening the aspects related to the robustness of nonuniform exponential dichotomy presented in \cite{Barreira2008-1}, which were discussed without considering admissibility.

An interesting issue is whether there exists the  exponential dichotomy in differential equations driven by nonlocal external force perturbations. Recently, Straatman and Hupkes \cite{Straatman2021} proved that a functional differential equations of mixed type with infinite range discrete and/or
continuous interactions admit exponential dichotomies that imposed the result of \cite{Paret1999}. In this nonlocal perturbation, we consider an example in $X=C_b(\R)$ with norm $\|x\|=\sup_{t\in \R}|x(t)|$ for the non-autonomous, integro-differential equation is given by
\begin{equation}\label{IDE}
x'(t)=a(t)x(t)+\int_{\R} g(t,\xi)x(t+\xi)d\xi,
\end{equation}
where  the functions $t\mapsto a(t)$ belongs to $X,$ $t\mapsto g(t,\cdot)$ belongs to $C_b(\R;L^1(\R))$. It is clear that $U(t,s)=\exp(\int_s^t a(\tau)d\tau)$ is the evolution family of the linear equation of (\ref{IDE}), and its admits an exponent dichotomy on $\R$ for projection $P^s(t)=1$ or $P^s(t)=0$. Let $A(t)=a(t):\R\to L(X)$, and $B(t):\R\to L(X)$ as
$$
B(t)x(t):=\int_{\R} g(t,\xi)x(t+\xi)d\xi,
$$
Hence, the  equation (\ref{IDE}) rewrites to
\begin{equation}\label{eq:2.2}
	x'(t)=A(t)x+B(t)x.
\end{equation}
By a simple calculation, taking $g(t,\xi)=w(t)J(\xi)$ as an oscillating slowly increasing kernel $J(\xi)=\sin(\xi^2)(1+\xi^2)^{-\beta}$, $\beta\in(1/2,1)$, but the requirement of $\|B(t)\|\leq \delta e^{-2\epsilon|t|} $ in \cite{Barreira2008-1} is not available if $w(t)\sim |t|^{-2\beta}e^{-2\epsilon|t|}$ for sufficiently small $\delta>0$.
By applying the  nonuniform exponential dichotomy results for a linear problem presented in \cite{Dragivcevic2022} with admissibility of a pair of function classes, and integrating them with more generalized integral condition, we establish the  roughness of equation (\ref{eq:2.2}).

In this paper,  the main objective is to
study the nonuniform exponential dichotomy in order to expand the set of admissible small perturbations, with applications to nonlocal integro-differential equations.
In order to achieve this goal, in Section 2 we introduce some notations and the main definitions. In Section 3, we prove the existence of nonuniform exponential dichotomy, under a perturbation with integral condition. We also propose example to illustrate the result.

\section{Preliminaries}

Let $X$ be a real Banach space, let $L(X)$ be the Banach algebra of all bounded linear operators from $X$ into itself. Denotes $J$ by a real interval of $\R_+$, $\R_-$ or $\R$. The norms on $X$ and $L(X)$ are denoted by both $\|\cdot\|$ for a convenient.  Let $\Delta_J=\{(t,s),t,s,\in J,t\geq s\}$.

We denotes a evolution family $\{U(t,s)\}_{t\geq s}$ by $U(t,s):\Delta_J\to L(X)$ generated with operator $A(t)$, i.e., the operator $U(t,t)=Id$, $I$ is the identity operator; $U(t,s)=U(t,r)U(r,s)$, for all $t\geq r\geq s,\in J$; $(t,s)\mapsto U(t,s)x$ is continuous for every $x\in X$. Let a strongly continuous function $P : J\to L(X)$ be a projection valued function if
$P^2(t)=P(t)$, for every $t\in J$. Denote by $Q(t)=Id-P(t)$ the complementary projection of $P(t)$, for every $t\in J$, and $N(P(s))$ is the null space of $P(s)$.

For the evolution family generated by linear operator $A(t)$, it is convenient to consider the Cauchy problem
\begin{equation}\label{mild}
\left\{	\begin{aligned}
& x'(t)=A(t)x(t)+y(t),\quad t\geq s,\in J,\\
&x(s)=x_s\in X,
	\end{aligned}\right.
\end{equation}
has a solution in the mild form
$$x(t)=U(t,s)x_s+\int_0^tU(t,\tau)y(\tau)d\tau,$$
for some $y(\cdot)\in X$,  see e.g. \cite{Pazy1983,Enagel1999}. Additionally, one is interested in
exponential dichotomy on invariant subspaces, see e.g. \cite{Henry}.
\begin{definition}
Let $U(t,s)$ be the evolution operator having a nonuniform exponential dichotomy, that is, there exists a projection $P(\cdot):J\to L(X)$ such that
\begin{enumerate}
	\item [{\rm (i)}] $U(t,s)P (s) =P (t)U (t,s)$ for $t\geq s$,
	\item [{\rm (ii)}] $U(t,s)|_{N(P(s))}:N(P(s))\to N(P(t))$ is invertible for $t\geq s$,
	\item [{\rm (iii)}] there exist constants $\alpha,K > 0$ and $\varepsilon\geq0$ such that
	$$
	\|U(t,s)P(s)\|\leq K e^{-\alpha(t-s)+\varepsilon|s|},~~t\geq s,
	$$
	$$
	\|U(t,s)Q(s)\|\leq Ke^{-\alpha(s-t)+\varepsilon|s|},~~s\geq t,
	$$
	where $U(t,s):(U(t,s)|_{N(P(s))})^{-1}$ for $s\geq t$.
\end{enumerate}
\end{definition}
Clearly, for $\varepsilon=0$, we say $U(t,s)$ have a uniform exponential dichotomy. In addition, the $\alpha$ and $Ke^{\epsilon|s|}$ are called the exponent and the bound of a nonuniform exponential dichotomy, respectively.

We next introduce the function classes for admissibility  introduced by \cite{Dragivcevic2022}. For each $\varepsilon\geq0$, $\beta>0$, let $X_{\varepsilon,\beta}(J)$ consist of all locally integrable $x:J\to X$ such that
\begin{align*}
	\|x\|_{\varepsilon,\beta}=&\sup_{t\in J}\left( e^{-\beta t}\int_{t}^{t+1}e^{\varepsilon|\tau|}\|x(\tau)\|d\tau\right),~~J=\R,~or~J=\R_+,~or\\	\|x\|_{\varepsilon,\beta}=&\sup_{t\in J}\left( e^{-\beta t}\int_{t-1}^{t}e^{\varepsilon|\tau|}\|x(\tau)\|d\tau\right),~~J=\R_-.
\end{align*}

Consider $X_{\varepsilon,\beta}(\R)$ consist of all locally integrable $x:\R\to X$ such that
\begin{align*}
	\|x\|_{\varepsilon,\beta|\cdot|}=&\sup_{t\in\R}\left( e^{-\beta |t|}\int_{t}^{t+1}e^{\varepsilon|\tau|}\|x(\tau)\|d\tau\right).
\end{align*}
Let $\Gamma_\beta$ be defined by
$$
\Gamma_\beta(J)=\{x\in X:~~\|x(\cdot)\|~\text{is continuous on}~J,~~\|x\|_\beta:=\sup _{t\in J}e^{-\beta t}\|x(t)\|<+\infty\},
$$
and $\Gamma_{\beta|\cdot|}$ be defined by
$$
\Gamma_{\beta|\cdot|}(\R)=\{x\in X:~~\|x(\cdot)\|~\text{is continuous on}~\R,~~\|x\|_{\beta|\cdot|}:=\sup _{t\in \R}e^{-\beta |t|}\|x(t)\|<+\infty\},
$$

\begin{definition}
	The pair $(Y,Z)$ is said to be (properly admissible) of function classes with respect to an evolution family $U (t,s)$ if for each $y\in Y$, there exists a (unique) $x\in Z$ such that
	$$
	x(t)=U (t,s)x(s)+ \int_s^t U (t,\tau)y(\tau)d\tau,~~t\geq s,\in J,
	$$
	
	In this case, the pair $(y,x)$ is (properly) admissible with respect to $U(t,s)$.
\end{definition}

\section{The existence of nonuniform exponential dichotomies}

In this section, we show that the problem (\ref{eq:2.2}) admits a nonuniform exponential dichotomies when $\{U(t,s)\}_{t\geq s,\in J}$ admits a nonuniform exponential dichotomy associated with $A(\cdot)\in L(X)$.
We first introduce the Green function $G$ by means of the nonuniform exponential dichotomy of $U(t,s)$, which is defined as
$$
G(t,s)=\left\{\begin{aligned}
	&U(t,s)P(s),&& t\geq s,\\
	&-U(t,s)Q(s), && t<s.
\end{aligned} \right.
$$

\begin{lemma} \label{Le3.2}
	Let
	$$
	x(t)=\int_J G(t,\tau)f(\tau)d\tau,~~J=\R_\pm,~or~J=\R,
	$$
	for some appropriate $f\in X$,
	then $x$ is a solution of
	\begin{align*}
		x(t)=&U(t,s)x(s) +\int_s^tU(t,\tau)f(\tau) d\tau,~~t\geq s,\in J .
	\end{align*}
\end{lemma}

\begin{proof}
	The cases of $\R_-$ and $\R$ are similar, we just show the proof of $\R_+$. In fact, for $t\geq s$,
	\begin{align*}
		x(t)-U(t,s)x(s)=&\int_0^{+\infty}G(t,\tau)f(\tau)d\tau-U(t,s)\int_0^{+\infty}G(s,\tau)f(\tau)d\tau\\
		=&\int_0^t U(t,\tau)P(\tau)f(\tau)d\tau-\int_t^{+\infty}U(t,\tau)Q(\tau)f(\tau)d\tau\\
		&-U(t,s)\int_0^s U(s,\tau)P(\tau)f(\tau)d\tau+U(t,s)\int_s^{+\infty}U(s,\tau)Q(\tau)f(\tau)d\tau\\
		=&\int_s^t U(t,\tau)P(\tau)f(\tau)d\tau -\int_t^{+\infty}U(t,\tau)Q(\tau)f(\tau)d\tau\\
		& + \int_s^{+\infty}U(t,\tau)Q(\tau)f(\tau)d\tau\\
		=&\int_s^t U(t,\tau)P(\tau)f(\tau)d\tau+\int_s^t U(t,\tau)Q(\tau)f(\tau)d\tau
		=\int_s^t U(t,\tau) f(\tau)d\tau.
	\end{align*}
	Thus, the conclusion holds.
\end{proof}

\begin{lemma}\label{Le3.1}
 	Let $U(t,s)$ admit a nonuniform exponential dichotomy with projection $P(s)$ and exponent $\alpha$, bound $Ke^{\epsilon|s|}$ such that $0\leq \epsilon<\alpha-\beta$, and let $b(t)=\|B(t)\|e^{\varepsilon |t|}$ satisfying
 	$$
 \theta:=\sup_{t\in\R} \int_{-\infty}^{+\infty} e^{-(\alpha-\beta)|t-\tau|}b(\tau)d\tau<+\infty.
 	$$
 Then for $y\in Y_{\epsilon,\beta}(J)$, the equation
\begin{equation}\label{Eq}
	x'(t)=(A(t)+B(t))x(t)+y(t),
\end{equation}
	has a solution in $\Gamma_\beta(J)$ if and only if there exists a function $x\in\Gamma_\beta(J)$ such that
	\begin{equation}\label{eqe}
	\begin{aligned}
		x(t)=&U(t,s)P(s)x_0+\int_s^tU(t,\tau)P(\tau)(B(\tau)x(\tau)+y(\tau))d\tau\\
		&-\int_t^\infty U(t,\tau)Q(\tau)(B(\tau)x(\tau)+y(\tau))d\tau,~~~t\geq s,\in J= \R_+,
	\end{aligned}
\end{equation}
or
	\begin{equation}\label{eqe2}
	\begin{aligned}
		x(t)=&U(t,s)Q(s)x_0-\int_{t}^s U(t,\tau)Q(\tau)(B(\tau)x(\tau)+y(\tau))d\tau\\
		&+\int_ {-\infty}^t U(t,\tau)P(\tau)(B(\tau)x(\tau)+y(\tau))d\tau,~~~t\in J= \R_-,
	\end{aligned}
\end{equation}
or a function $x\in \Gamma_{\beta }(J)$ for $y\in Y_{\epsilon,\beta }(J)$~or $x\in \Gamma_{\beta|\cdot|}(J)$  for $y\in Y_{\epsilon,\beta|\cdot|}(J)$ as
	\begin{equation}\label{eqe3}
	\begin{aligned}
		x(t)=& \int_{-\infty}^\infty G(t,\tau) (B(\tau)x(\tau)+y(\tau))d\tau ,~~~t\in J= \R ,
	\end{aligned}
\end{equation}
for some $x_0\in X$.
\end{lemma}
\begin{proof}
For $J=\R_+$, one see  that if $x\in\Gamma_\beta(J)$ satisfies (\ref{eqe}), and then $x$ is a solution of (\ref{Eq}) in $\Gamma_\beta(J)$ clearly. In fact, since $t\mapsto U(t,s)$ is differentiable as $\partial_t U(t,s)=A(t)U(t,s)$ on $L(X)$,
this yields
\begin{equation}\label{equa}
\begin{aligned}
x'(t)=&A(t)U(t,s)P(s)x_0+\int_s^tA(t)U(t,\tau)P(\tau)(B(\tau)x(\tau)+y(\tau))d\tau\\
&-\int_t^\infty A(t) U(t,\tau)Q(\tau)(B(\tau)x(\tau)+y(\tau))d\tau+B(t)x(t)+y(t)\\
=&A(t)x(t)+B(t)x(t)+y(t).
\end{aligned}
\end{equation}

On the contrary, let $y\in Y_{\epsilon,\beta}(J)$ such that equation (\ref{Eq}) has a solution $x\in\Gamma_\beta(J)$, by the variation of constants formula, the solution of linearly perturbation equation can be represented as
	\begin{align*}
		x(t)= U(t,s)x(s) +\int_s^tU(t,\tau) (B(\tau)x(\tau)+y(\tau))d\tau ,
	\end{align*}
	it follows that
	\begin{align*}
		P(t)x(t)= P(t)U(t,s)x(s) +\int_s^tP(t)U(t,\tau) (B(\tau)x(\tau)+y(\tau))d\tau ,
	\end{align*}
	and
	\begin{align*}
		Q(t)x(t)= U(t,s)Q(s)x(s) +\int_s^tU(t,\tau)Q(\tau) (B(\tau)x(\tau)+y(\tau))d\tau ,
	\end{align*}
	this implies that
	\begin{align*}
		Q(s)x(s)= U(s,t)Q(t)x(t)-\int_s^tU(s,\tau)Q(\tau) (B(\tau)x(\tau)+y(\tau))d\tau ,
	\end{align*}
	by the dichotomy estimate for $0\leq\epsilon<\alpha-\beta$
	\begin{align*}
		\|U(s,t)Q(t)x(t)\|\leq&  Ke^{-\alpha(t-s)+\epsilon|t|}\|x(t)\|
		\leq Ke^{-\alpha(t-s)+\epsilon|t|+\beta t}\|x \|_\beta\\
		\to& 0,\quad {\rm as}\ t\to+\infty.
	\end{align*}
	On the other hand, for all $t\geq s$, we have
	\begin{align*}
		&\left\|\int_s^tU(s,\tau)Q(\tau) (B(\tau)x(\tau)+y(\tau))d\tau\right\|\\
		\leq& K\int_s^t e^{-\alpha(\tau-s)+\epsilon|\tau|} ( b(\tau)e^{-\epsilon|\tau|}\|x(\tau)\|+\|y(\tau)\|)d\tau\\
		\leq&K \int_s^t e^{-\alpha(\tau-s) } b(\tau) e^{\beta\tau} \|x \|_\beta d\tau+
		K\int_s^t e^{-\alpha(\tau-s)+\epsilon|\tau|} \|y(\tau)\| d\tau.
	\end{align*}
	Since
	\begin{align*}
	\int_s^t e^{-\alpha(\tau-s) } b(\tau)e^{\beta\tau} d\tau
	= e^{\beta s}\int_s^t e^{-(\alpha-\beta)(\tau-s) } b(\tau)d\tau,
\end{align*}
	and
	\begin{align*}
		\int_s^t e^{-\alpha(\tau-s)+\epsilon|\tau|} \|y(\tau)\| d\tau\leq &\sum_{k=0}^{[t-s]}\int_{s+k}^{s+k+1}e^{-\alpha(\tau-s)+\epsilon|\tau|} \|y(\tau)\| d\tau\\
		\leq& e^{\alpha s}\sum_{k=0}^{[t-s]}e^{-(\alpha-\beta)(s+k)}\|y\|_{\epsilon,\beta}
		= e^{\beta s}\sum_{k=0}^{[t-s]}e^{-(\alpha-\beta) k}\|y\|_{\epsilon,\beta}.
	\end{align*}
This implies that
	\begin{align*}
	 \left\|\int_s^tU(s,\tau)Q(\tau) (B(\tau)x(\tau)+y(\tau))d\tau\right\|
		\leq& K\theta e^{\beta s} \|x \|_\beta+
		Ke^{\beta s}\|y\|_{\epsilon,\beta}\frac{1-e^{-(\alpha-\beta)([t-s]+1)}}{1-e^{-(\alpha-\beta)}} ,
	\end{align*}
	where we use the identity
	$$
	\sum_{k=0}^{[t-s]}e^{-(\alpha-\beta) k}
	=\frac{1-e^{-(\alpha-\beta)([t-s]+1)}}{1-e^{-(\alpha-\beta)}}.
	$$
	Thus, let $t\to+\infty$, we have
	\begin{align*}
		Q(s)x(s)= -\int_s^{+\infty} U(s,\tau)Q(\tau) (B(\tau)x(\tau)+y(\tau))d\tau ,
	\end{align*}
	this shows
	\begin{align*}
		Q(t)x(t)= &-U(t,s)\int_s^{+\infty}U(s,\tau)Q(\tau) (B(\tau)x(\tau)+y(\tau))d\tau \\ &+\int_s^tU(t,\tau)Q(\tau) (B(\tau)x(\tau)+y(\tau))d\tau\\
		=&-\int_t^{+\infty}U(t,\tau)Q(\tau) (B(\tau)x(\tau)+y(\tau))d\tau.
	\end{align*}
	Hence, it follows that
	\begin{align*}
		x(t)=&	P(t)x(t)+Q(t)x(t)\\
		=& P(t)U(t,s)x(s) +\int_s^tP(t)U(t,\tau) (B(\tau)x(\tau)+y(\tau))d\tau \\
		&-\int_t^{+\infty}U(t,\tau)Q(\tau) (B(\tau)x(\tau)+y(\tau))d\tau.
	\end{align*}
	Clearly, the solution $x$ of linearly perturbation equation satisfies above equation for each $x_0\in X$ such that $P(s)x_0=P(s)x(s)$. The case of $J=\R_-$ is similar.
	
It is from the case of $\R_+$ and $\R_-$, similarly to (\ref{equa}), if $\Gamma_{\beta|\cdot|}(\R)$ satisfies (\ref{eqe3}), and then $x$ is a solution of (\ref{Eq}) in $\Gamma_{\beta|\cdot|}(\R)$. On the contrary, let $x$ be a solution of (\ref{Eq}) in $\Gamma_{\beta|\cdot|}(\R)$ for $y\in Y_{\epsilon,\beta|\cdot|}(\R)$, in particular, when $y\in Y_{\epsilon,\beta }(\R_+)$, it follows $x$ in  $\Gamma_{\beta }(\R_+)$ and another case is the same for $\R_-$. From the aforementioned proved results for $\R_+$ and $\R_-$, there are $x_+$ and $x_-$ such that
	\begin{equation}\label{x+}
	\begin{aligned}
		x(t)=&U(t,0)P(0)x_++\int_0^tU(t,\tau)P(\tau)(B(\tau)x(\tau)+y(\tau))d\tau\\
		&-\int_t^\infty U(t,\tau)Q(\tau)(B(\tau)x(\tau)+y(\tau))d\tau,
	\end{aligned}
\end{equation}
and
\begin{equation*}
	\begin{aligned}
		x(t)=&U(t,0)Q(0)x_--\int_{t}^0 U(t,\tau)Q(\tau)(B(\tau)x(\tau)+y(\tau))d\tau\\
		&+\int_ {-\infty}^t U(t,\tau)P(\tau)(B(\tau)x(\tau)+y(\tau))d\tau.
	\end{aligned}
\end{equation*}
Let $t=0$ in above equations, it follows
	\begin{equation*}
	\begin{aligned}
  P(0)x_+&- \int_0^\infty U(0,\tau)Q(\tau)(B(\tau)x(\tau)+y(\tau))d\tau
 =  Q(0)x_- \\& +\int_ {-\infty}^0 U(0,\tau)P(\tau)(B(\tau)x(\tau)+y(\tau))d\tau.
	\end{aligned}
\end{equation*}
Since $P(0)Q(0)=0$, $P(0)^2=P(0)$ and $P(0)U(0,\tau)=U(0,\tau)P(\tau)$, $U(0,\tau)Q(\tau)=Q(0)U(0,\tau)$, we obtain
	\begin{equation*}
		P(0)x_+= \int_ {-\infty}^0 U(0,\tau)P(\tau)(B(\tau)x(\tau)+y(\tau))d\tau.
\end{equation*}
Substituting $P(0)x_+$ into (\ref{x+}), this yields
	\begin{equation*}
	\begin{aligned}
		x(t)=&U(t,0)\int_ {-\infty}^0 U(0,\tau)P(\tau)(B(\tau)x(\tau)+y(\tau))d\tau+\int_0^tU(t,\tau)P(\tau)(B(\tau)x(\tau)+y(\tau))d\tau\\
		&-\int_t^\infty U(t,\tau)Q(\tau)(B(\tau)x(\tau)+y(\tau))d\tau\\
		=&\int_ {-\infty}  ^tU(t,\tau)P(\tau)(B(\tau)x(\tau)+y(\tau))d\tau
	 -\int_t^\infty U(t,\tau)Q(\tau)(B(\tau)x(\tau)+y(\tau))d\tau,
	\end{aligned}
\end{equation*}
which means that the solution $x$ satisfies (\ref{eqe3}). The case of $\R_-$ is similar.
The proof is completed.

\end{proof}

We need that the following lemma to establish the nonuniform exponential dichotomy on $\R$ in \cite[Theorem 3]{Dragivcevic2022}.
\begin{lemma}\label{thm1}
	Let $U(t,s)$ be an evolution family, $\beta>0$ and $\epsilon\geq0$. Suppose that the pairs $(Y_{\epsilon,\beta}(\R),\Gamma_{\beta }(\R))$, $(Y_{\epsilon,-\beta}(\R),\Gamma_{-\beta }(\R))$  and $(Y_{\epsilon,\beta|\cdot|}(\R),\Gamma_{ \beta|\cdot| }(\R))$ are properly admissible with respect to $U(t,s)$. Then $U(t,s)$ admits a nonuniform exponential dichotomy on $\R$ with exponent $\beta$ and bound $Ke^{\epsilon |s|}$.
\end{lemma}
\begin{theorem}\label{thm3.1}
	Let $U(t,s)$ admit a nonuniform exponential dichotomy with projection $P(s)$, exponent $\alpha$  and bound $Ke^{\epsilon|s|}$ such that $0\leq \epsilon<\alpha-\beta$. Let $b(t)=\|B(t)\|e^{\varepsilon |t|}$ satisfying
\begin{equation}\label{condition}
	\theta =\sup_{t\in\R} \int_{-\infty}^{+\infty} e^{-(\alpha-\beta)|t-\tau|}b(\tau)d\tau<1/K.
\end{equation}
 Then (\ref{eq:2.2}) admits a nonuniform exponential dichotomy on $\R$.
\end{theorem}

\begin{proof}

We first  consider a space
$$
\mathscr 	C=\{u:\Delta_J\to L(X),~~U~\text{is continuous} \},
$$
with norm
$$
\|U\|_*=\sup_{(t,s)\in\Delta_J}\|U(t,s)\|e^{\beta(t-s)-\epsilon|s|}<+\infty.
$$
It is observe that $\mathscr C$ is a Banach space.
Set an operator $\mathscr F$ on $\mathscr C$ by
$$
\mathscr F U_B(t,s)=U(t,s)P(s)+\int_s^\infty G(t,\tau)B(\tau)U_B(\tau,s)d\tau.
$$
 Moreover, we get
\begin{align*}
	\|\mathscr FU_B(t,s)\|\leq & Ke^{-\alpha(t-s)+\epsilon|s|}+K \int_s^t e^{-\alpha(t-\tau) }b(\tau) \|U_B(\tau,s)\|d\tau\\
	&+\int_t^\infty e^{-\alpha( \tau-t) }b(\tau) \|U_B(\tau,s)\|d\tau\\
	\leq&Ke^{-\alpha(t-s)+\epsilon|s|}+Ke^{-\beta(t-s)+\epsilon|s|} \int_s^t e^{-(\alpha-\beta)(t-\tau) }b(\tau) d\tau \|U_B\|_*\\
	&+Ke^{-\beta(t-s)+\epsilon|s|} \int_ t^\infty e^{-(\alpha-\beta)( \tau-t) -2\beta(\tau-t)}b(\tau) d\tau \|U_B\|_*,
\end{align*}	
This shows that
\begin{align}\label{inequa}
	\|\mathscr FU_B \|_*
	\leq Ke^{-(\alpha-\beta)(t-s) }+K\theta\|U_B\|_*.
\end{align}	
Therefore, $\mathscr T:\mathscr C\to\mathscr  C$ is well defined. For any $U_B',U_B''\in\mathscr  C$, it yields
\begin{align*}
\|\mathscr F U_B''(t,s)-\mathscr F U_B'(t,s)\|\leq &\int_s^t\| U(t,\tau)P(\tau)B(\tau)(U_B''(\tau,s)-  U_B'(t,s))\|d\tau\\
&+\int_t^\infty \| U(t,\tau)Q(\tau)B(\tau)(U_B''(\tau,s)-  U_B'(t,s))\|d\tau\\
\leq&
K\int_s^t e^{-\alpha(t-\tau) }b(\tau)e^{-\beta(\tau-s)+\epsilon |s|} d\tau \|U_B'' -  U_B'\|_*\\
&+K\int_t^\infty e^{-\alpha(\tau-t) }b(\tau)e^{-\beta(\tau-s)+\epsilon |s|} d\tau \|U_B'' -  U_B'\|_*\\
=& K  e^{-\beta(t-s)+\epsilon|s|} \int_s^\infty e^{-(\alpha-\beta)|t-\tau| }b(\tau) d\tau\|U_B'' -  U_B'\|_*,
\end{align*}
this implies $
	\|\mathscr F U_B'' -\mathscr F U_B' \|_* \leq   K\theta  \|U_B'' -  U_B'\|_*, $
which is
a contraction in view of $K\theta<1$. Hence, there is a unique $U_B\in \mathscr C$ such that $\mathscr F U_B=U_B$ and
\begin{equation}\label{identity2}\begin{aligned}
	U_B(t,s)&=U(t,s)P(s)+\int_s^t U(t,\tau)P(\tau)B(\tau)U_B(\tau,s)d\tau\\ &-\int_t^\infty U(t,\tau)Q(\tau)B(\tau)U_B(\tau,s)d\tau.
\end{aligned}\end{equation}

Therefore, since $t\mapsto U(t,s)$ is differentiable in $L(X)$, for some $U_B\in\mathscr C$, $t\mapsto U_B(t,s)$ is differentiable in $L(X)$,  we have
\begin{equation*}\label{identity22}\begin{aligned}
	\partial_t 	U_B(t,s)=&A(t)U(t,s)P(s)+A(t)\int_s^t U(t,\tau)P(\tau)B(\tau)U_B(\tau,s)d\tau\\ &-A(t)\int_t^\infty U(t,\tau)Q(\tau)B(\tau)U_B(\tau,s)d\tau\\
	&+P(t)B(t)U_B(t,s)+Q(t)B(t)U_B(t,s)
	= (A(t)+B(t))U_B(t,s),
\end{aligned}\end{equation*}
which shows that $t\mapsto U_B(t,s)x$, $t\geq s$, is a solution of
(\ref{eq:2.2}) for every $x\in X$.
From the results in \cite{Dragivcevic2022}, we note that
$$
\|U(t,s)v\|\leq Ke^{-\alpha(t-s){\rm sgn}(t-s)+\epsilon|s|}\|v\|,~~\forall t,s\in\R,~~v\in S(s),~or~v\in U(s),
$$
where $S(s)$ is the stable subspace and $U(s)$ is the unstable subspace such that $X=S(s)+U(s)$,~$\forall s\in \R$.
 Hence, by using the aforementioned arguments, the solution of (\ref{identity2}) is equivalent to
\begin{equation}\label{identity}
	U_B(t,s)=U(t,s)+\int_s^t U(t,\tau)B(\tau)U_B(\tau,s)d\tau,
\end{equation}
for $t\geq s,\in \R$. It is clear from (\ref{identity}) that $U_B(t,t) = U(t,t)  =I$.
A similar statement of operator $\mathscr F$ shows that the equation
$$
f(t)=\int_r^tU(t,\tau)B(\tau)f(\tau)d\tau,
$$
just admits a zero solution for every $t\in \R$ and some $f\in\mathscr C$ because of $f(r)=0$. By the evolution property of $U(t,s)$, a simple calculation shows that $f(t)=U_B(t,\tau)U_B(\tau,s)-U_B(t,s)$ fulfills the above equation for every $t\geq\tau\geq s,\in \R$.
Hence, one can sufficiently verify $U_B(t,\tau)U_B(\tau,s)=U_B(t,s)$ for $t\geq\tau\geq s,\in \R$.	
We conclude that $\{U_B(t,s)\}_{t\geq s,\in \R}$ is an evolution family of
(\ref{eq:2.2}).

In the sequel, our strategy is to use Lemma \ref{thm1} to give the existence of a nonuniform exponential dichotomy for evolution family $U_B(t,s)$. Therefore, we need to verify the assumptions both satisfied, that is, the pair ($Y_{\epsilon,\beta}(\R),\Gamma_{\beta}(\R)$), ($Y_{\epsilon,-\beta}(\R),\Gamma_{-\beta}(\R)$) and ($Y_{\epsilon,\beta|\cdot|}(\R),\Gamma_{\beta|\cdot|}(\R)$) are both properly admissible with respect to $U_B(t,s)$.

It suffices to prove that for each $y\in Y_{\epsilon,\beta}(\R)$, there exists a unique function $x\in \Gamma_{\beta }(\R)$ such that
	(\ref{eqe3}) holds, because the other two cases can use the similar arguments ot obtain the desired conclusions.
	For given $y\in Y_{\epsilon,\beta}(\R)$, define a mapping
	$\phi_y:\Gamma_{\beta }(\R)\to \Gamma_{\beta }(\R)$ as
	\begin{align*}
		\phi_yx(t)=&\int_{-\infty}^{+\infty}G(t,\tau)(B(\tau)x(\tau)+y(\tau))d\tau,
	\end{align*}
where $G$ is the Green function associated the nonuniform exponential dichotomy
	possessing the dichotomy estimation
	$$
	\|G(t,s)\|\leq Ke^{-\alpha|t-s|+\epsilon|s|},~~\forall t,s\in\R,
	$$
	This mapping is well defined because for $x\in \Gamma_{\beta}(\R)$ and
	\begin{align*}
		\|\phi_yx(t)\|\leq &\left(\int_{-\infty}^{t}+\int_t^{+\infty}\right)\| G(t,\tau)\|( \|B(\tau)x(\tau)\|+\|y(\tau)\|)d\tau\\
		\leq &K\int_{-\infty}^{t}e^{-\alpha(t-\tau)+\epsilon|\tau|} ( b(\tau)e^{-\epsilon|\tau|}\|x(\tau)\|+\|y(\tau)\|)d\tau\\ &+K\int_{t}^{+\infty} e^{-\alpha(\tau-t)+\epsilon|\tau|} ( b(\tau)e^{-\epsilon|\tau|}\|x(\tau)\|+\|y(\tau)\|)d\tau
		=I_1+I_2.
	\end{align*}
	As for $I_1$, we have
	\begin{align*}
		I_1 \leq &K \int_{-\infty}^{t}e^{-\alpha(t-\tau) }   b(\tau)e^{\beta  \tau }\|x\|_\beta d\tau
		+K\int_{-\infty}^{t}e^{-\alpha(t-\tau)+\epsilon|\tau|} \|y(\tau)\| d\tau\\
		\leq& K e^{\beta t}\int_{-\infty}^{t}  e^{-( \alpha-\beta) (t-  \tau) }b(\tau) d\tau \|x\|_\beta
		+Ke^{- \alpha t}\sum_{k=-\infty}^{[t]}\int_k^{k+1}e^{\alpha\tau+\epsilon|\tau|}\|y(\tau)\|d\tau.
	\end{align*}
	Since
	$$\int_{-\infty}^{t}  e^{- ( \alpha-\beta) (t-  \tau) }b(\tau) d\tau\leq \theta,$$
	 and
	\begin{align*}\textsf{}
		\sum_{k=-\infty}^{[t]}e^{\alpha(k+1)+\beta k} =& \frac{e^{\alpha+(\alpha+\beta) [t]}}{1-e^{-(\alpha+\beta)}}\leq \frac{e^{(\alpha+\beta) t+\alpha} }{1-e^{-(\alpha+\beta)}},
	\end{align*}
	we have
	\begin{align*}
		I_1 \leq & Ke^{\beta t}\theta  \|x\|_\beta +Ke^{- \alpha t}\sum_{k=-\infty}^{[t]}e^{ \alpha (k+1)+\beta k} \|y\|_{\epsilon,\beta}
		\leq   Ke^{\beta t}\theta   \|x\|_\beta + K \|y\|_{\epsilon,\beta} \frac{e^{\alpha+\beta t} }{1-e^{-(\alpha+\beta)}}.
	\end{align*}
	Furthermore,
	\begin{align*}
		I_2\leq&K \int_{t}^{+\infty}e^{-\alpha(\tau-t) }  b(\tau) e^{\beta  \tau }\|x\|_\beta d\tau  +K\int_{t}^{+\infty} e^{-\alpha(\tau-t)+\epsilon|\tau|}  \|y(\tau)\| d\tau\\
		\leq &e^{\beta t}\theta\|x\|_\beta   +K\sum_{k=0}^\infty\int_{t+k}^{t+k+1} e^{-\alpha(\tau-t)+\epsilon|\tau|}  \|y(\tau)\| d\tau,
	\end{align*}
since
	\begin{align*}
		\sum_{k=0}^\infty\int_{t+k}^{t+k+1} e^{-\alpha(\tau-t)+\epsilon|\tau|}  \|y(\tau)\| d\tau\leq &e^{\beta t } \sum_{k=0}^\infty e^{-(\alpha-\beta) k} \|y\|_{\epsilon,\beta}=\frac{e^{\beta t }}{1-e^{-(\alpha-\beta)}}\|y\|_{\epsilon,\beta}.
	\end{align*}
	Therefore, we obtain
	\begin{align*}
		I_2\leq& Ke^{\beta t}\theta\|x\|_\beta   + \frac{Ke^{\beta t }}{1-e^{-(\alpha-\beta)}}\|y\|_{\epsilon,\beta}.
	\end{align*}
	Together these arguments,
	\begin{align*}
		\|\phi_yx \|_\beta\leq &2K\theta\|x\|_\beta  + K\|y\|_{\epsilon,\beta}\left(\frac{e^\alpha}{1-e^{-(\alpha+\beta) }}+\frac{1}{1-e^{-(\alpha-\beta)}}\right).
	\end{align*}
	This implies that $\phi_y(\Gamma_{\beta}(\R))\subset \Gamma_{\beta}(\R)$. For arbitrary $u, v\in \Gamma_{\beta}(\R)$,
	\begin{align*}
		\|\phi_yu(t)-\phi_yv(t)\|\leq &\int_{-\infty}^{+\infty}\| G(t,\tau)\|  \|B(\tau)(u(\tau)-v(\tau))\| d\tau\\
		\leq & K \int_{-\infty}^{+\infty}e^{-\alpha|t-\tau|+\beta\tau}b(\tau)  d\tau\|u-v\|_\beta\\
		\leq & K\theta  e^{\beta t}  \|u-v\|_\beta,
	\end{align*}
	we thus have
	\begin{align*}
		\|\phi_yu -\phi_yv \|_\beta
		\leq &   K\theta \|u-v\|_\beta.
	\end{align*}
	Thus, $\phi_y$ is a contraction in the Banach space $\Gamma_{\beta}(\R)$ for $  K\theta <1$. Therefore, $\phi_y$ has a unique fixed point $x\in \Gamma_{\beta}(\R)$ such that
\begin{equation}\label{equ-i}
	x_y(t)=\phi_y x_y(t)=\int_{-\infty}^{+\infty}G(t,\tau)(B(\tau)x_y(\tau)+y(\tau))d\tau.
\end{equation}
	By lemma \ref{Le3.1}, the function $x_y$ is a unique solution of
\begin{equation}\label{equ-id}
	x'=(A(t)+B(t))x+y
\end{equation}
in $ \Gamma_{\beta}(\R)$. In particular, by the variation of constants formula, the solution of (\ref{equ-id}) is given by
\begin{equation}\label{equ-idd}
	x(t)=U_B(t,s)x(s)+\int_s^t U_B(t,\tau)y(\tau)d\tau.
\end{equation}
In fact,
since $t\mapsto U_B(t,s)$ is differentiable as $\partial_tU_B(t,s)=(A(t)+B(t))U_B(t,s)$ in $L(X)$,
if (\ref{equ-idd}) holds for $x_y\in \Gamma_{\beta}(\R)$ associated with $y\in Y_{\epsilon,\beta}(\R)$,
we obtain that
\begin{align*}
	x'(t)=& (A(t)+B(t))U_B(t,s)x(s)+(A(t)+B(t))\int_s^t U_B(t,\tau) y(\tau)d\tau+y(t)\\
	=&(A(t)+B(t))x(t)+y(t).
\end{align*}
It means that the solution of (\ref{equ-idd}) is that of (\ref{equ-id}), and then it is also a unique solution of (\ref{equ-i}) by the uniqueness. Thus, $(Y_{\epsilon,\beta}(\R),   \Gamma_{\beta}(\R))$ is admissible. The proof is complete.

\end{proof}

Let $J=\R_+$. For a closed subspace $Z$ of $X$, we define
$$\Gamma_{ \beta,Z}=\{x\in \Gamma_{\beta}(J):~x(0)\in Z\}.$$
Then $\Gamma_{ \beta,Z}$ is a closed subspace of $(\Gamma_{\beta}(J),\|\cdot\|_\beta)$.
We next use the following result to establish the nonuniform exponential dichotomy on $\R_+$ in \cite[Theorem 4]{Dragivcevic2022}.
\begin{lemma}\label{thm2}
	Let $U(t,s)$ be evolution family, $Z$ be a closed subspace of $X$, and $\beta>0$ and $\epsilon\geq0$. Suppose that the pairs $(Y_{\epsilon,\beta}(J),\Gamma_{\beta,Z}(J))$  and $(Y_{\epsilon,-\beta}(J),\Gamma_{-\beta,Z}(J))$ are properly admissible with respect to $U(t,s)$, for $J=\R_+$. Then $U(t,s)$ admits a nonuniform exponential dichotomy on $J$ with the bound $Ke^{\epsilon |t|}$.
\end{lemma}

\begin{theorem}\label{the3.2}
	Let $U(t,s)$ admit a nonuniform exponential dichotomy with exponent $\alpha$, bound $Ke^{\epsilon|s|}$ such that $0\leq \epsilon<\alpha-\beta$, and $Z$ be a closed subspace of $X$.  Let $b(t)=\|B(t)\|e^{\varepsilon |t|}$ satisfying (\ref{condition}).
Then (\ref{eq:2.2}) admits a nonuniform exponential dichotomy on $\R_+$.
\end{theorem}

\begin{proof}
Let $J=\R_+$, a similar argument in Theorem \ref{thm3.1}, by the variation of constants formula, there is a unique $U_B\in\mathscr C$ such that
(\ref{identity}) fulfills, and then it is an evolution family and $U_B(t,s)y$ is a solution of (\ref{eq:2.2}) for all $y\in X$. By the strategy is to use Lemma \ref{thm2} to give verify  that is the pairs ($Y_{\epsilon,\beta}(J),\Gamma_{\beta,Z}(J)$) is properly admissible with respect to $U_B(t,s)$, and the case of ($Y_{\epsilon,-\beta}(J),\Gamma_{-\beta,Z}(J)$) follows from the same way.

For given $y\in Y_{\epsilon,\beta}(J)$, define a mapping
$\phi_y:\Gamma_{\beta,Z}(J)\to \Gamma_{\beta,Z}(J)$ as
\begin{align*}
	\phi_yx(t)=&U(t,0)P(0)x(0)+\int_{0}^{+\infty}G(t,\tau)(B(\tau)x(\tau)+y(\tau))d\tau.
\end{align*}
It follows from the proof of Theorem \ref{thm3.1} that $\phi_y$ has a unique fixed point $x\in \Gamma_{\beta,Z}(J)$ such that
$$
x_y(t)=\phi_y x_y(t)=U(t,0)P(0)x(0)+\int_{0}^{+\infty}G(t,\tau)(B(\tau)x_y(\tau)+y(\tau))d\tau.
$$
By Lemma \ref{Le3.1}, the function $x_y$ is a unique solution of $x'=(A(t)+B(t))x+y$
in $ \Gamma_{\beta,Z}(J)$, which implies that $(Y_{\epsilon,\beta}(J),   \Gamma_{\beta,Z}(J))$ is admissible with respect to $U_B$. The proof is complete.
	
\end{proof}

It is notice that if the closed subspace of $X$
	$$
W(0)=\{v\in X:~~U(t,s)~\text{has a bounded negative continuation at}~(0,v)\},
	$$
that is, $W(0)$ is the set of all $v\in X$ such that there exists $x:(-\infty,0]\to X$ continuous and $x(0)=v$, $x(t_1)=U(t_1,t_2)x(t_2)$ for $t_2\leq t_1\leq 0$, $\sup_{t\leq 0}\|x(t)\|<+\infty$., such that $X=Z+W(0)$ and $U(t,s)$ is not invertible in the case of $J=\R_-$,  \cite[Theorem 4]{Dragivcevic2022} also shows that  (\ref{eq:2.2}) admits nonuniform exponential dichotomy on $\R_-$. This allows us to provide the following fact.
 \begin{remark}

By the same method as in Theorem \ref{thm3.1}, there exists a unique evolution family $U_B$ satisfying (\ref{identity}). Moreover, by direct calculation, $U_B$ has a bounded negative continuation at $(0,v)$ since $U\in W(0)$, say it as $W_B(0)$. Although it remains unclear whether $U_B(t,s)$ is invertible, under the assumption that $U(t,s)$ is not invertible on $J=\mathbb{R}_-$, we may suppose that $X=Z+W_B(0)$ and that $U_B$ is also not invertible. Then, according to Theorem \ref{the3.2}, a similar argument shows that (\ref{eq:2.2}) admits a nonuniform exponential dichotomy on $\mathbb{R}_-$ provided with the condition (\ref{condition}).

\end{remark}

\begin{example}
Consider an example which satisfies the conditions of our theorem in $X=\R^2$ with norm $\|(x_1,x_2)\|=|x_1|+|x_2|$, the
system
\begin{equation}\label{Sys}
\left\{ \begin{aligned}
	x_1'=&(-1+\cos t)x_1+e^{-2\epsilon t} I_t^\gamma \sin t~ x_1,\\
	x_2'=&(1+\sin t)x_2+e^{-2\epsilon t} I_t^\gamma \cos t~ x_2,
\end{aligned}\right.
\end{equation}
where $t\geq0$, $\epsilon>0$ and $I^\gamma_t$ is the nonlocal integral as the Riemann-Liouville fractional integral of order $\gamma>0$ by
$$
I^\gamma_t v(t)=\frac{1}{\Gamma(\gamma)}\int_0^t(t-s)^{\gamma-1}v(s)ds,~~v\in \R,~~t\geq0.
$$
Now, let functions $a,b,c,d:\R_+\to \R$ be given by
$a(t)=-t+\sin t$, $b(t)=t-\cos t$, $c(t)=e^{-2\epsilon t}I^\gamma_t \sin t$ and $d(t)=e^{-2\epsilon t}I^\gamma_t\cos t$. Consider $A(t):\R_+\to L(X)$ and $B(t):\R_+\to L(X)$  given by
$$
A(t)(x_1,x_2)=(a'(t)x_1,b'(t)x_2),~~~~B(t)(x_1,x_2)=(c(t)x_1,d(t)x_2),
$$	
obviously, this system  fulfills the differential equation (\ref{eq:2.2}). By a simple calculation, $\|B(t)\|\leq c_\gamma e^{-2\epsilon t} t^{\gamma} $ for $t\geq 0$ since $\|B(t)(x_1,x_2)\|\leq  c_\gamma e^{-2\epsilon t} t^{\gamma}\|(x_1,x_2)\|$ with $c_\gamma=1/\Gamma(\gamma+1)$.
However, this perturbation is not suitable for the requirement proposed in Barreira and Valls \cite{Barreira2008-1}, where
their established the nonuniform exponential dichotomy if $\|B(t)\|\leq \delta e^{-2\epsilon|t|}$ for $t\in J$ and $\delta>0$ sufficiently small, but our Theorem \ref{the3.2} is applicable for a closed subspace $\{0,0\}$, a line through zero point or $\R^2$,
and the evolution family of linear differential equation is $\{U(t,s)\}_{t\geq s\geq0}$ as
$$
	U(t,s)(x_1,x_2)=(e^{a(t)-a(s)}x_1,e^{b(t)-b(s)}x_2),~~(x_1,x_2)\in X,~~t\geq s\geq 0.
	$$
Hence, the system (\ref{Sys}) admits a nonuniform exponential dichotomy being of roughness under perturbation.
\end{example}	
	
	\section*{Acknowledgements}
The authors wish to express their sincere gratitude to Prof. Linfeng Zhou for his valuable comments and constructive suggestions.

\end{document}